\documentclass[12pt]{amsart}
\usepackage{amsmath}
\usepackage{amsfonts}
\usepackage{amssymb}
\makeatletter
\def\@seccntformat#1{%
  \protect\textup{%
    \protect\@secnumfont
    \expandafter\protect\csname format#1\endcsname 
    \csname the#1\endcsname
    \protect\@secnumpunct
  }%
}

\begin{document}
\title[Schur polynomials in all primitive $n$th roots of unity]
{The Schur polynomials \\
in all primitive $n$th roots of unity}
\author[M. Hidaka]{Masaki Hidaka}
\address{Graduate School of Science and Engineering, 
Kagoshima University, Kagoshima 890-0065, Japan}
\email{k3575488@gmail.com}
\author[M. Itoh]{Minoru Itoh}
\address{Graduate School of Science and Engineering, 
Kagoshima University, Kagoshima 890-0065, Japan}
\email{itoh@sci.kagoshima-u.ac.jp }
\date{}
\begin{abstract}
We show that the Schur polynomials in all primitive $n$th roots of unity 
are $1$, $0$, or $-1$, 
if $n$ has at most two distinct odd prime factors.
This result can be regarded as a generalization of properties
of the coefficients of the cyclotomic polynomial and its multiplicative inverse.
The key to the proof is the concept of a unimodular system of vectors.
Namely, this result can be reduced to the unimodularity of the tensor product of two maximal circuits
(here we call a vector system a maximal circuit,
if it can be expressed as $B \cup \{ -\sum B\}$ with some basis $B$).
\end{abstract}
\thanks{This research was partially supported by JSPS Grant-in-Aid for Scientific Research (C) 16K05067 and 21K03209.
}
\maketitle
\theoremstyle{plain}
   \newtheorem{theorem}{Theorem}[section]
   \newtheorem{proposition}[theorem]{Proposition}
   \newtheorem{lemma}[theorem]{Lemma}
   \newtheorem{corollary}[theorem]{Corollary}
\theoremstyle{definition}
   \newtheorem{definition}[theorem]{Definition}
   \newtheorem{example}[theorem]{Example}
\theoremstyle{remark}
   \newtheorem*{remark}{Remark}
   \newtheorem*{remarks}{Remarks}
\numberwithin{equation}{section}
%

%
\section{Introduction}
%
%
The following assertion on
the Schur polynomials in all primitive $n$th roots of unity 
is the main theorem of this article\footnote{%
This theorem was first found in the Master's thesis of the first author \cite{H}.}:

\begin{theorem}\label{thm:main}\slshape
   Let $\omega_1,\ldots,\omega_d$ be all primitive $n$th roots of unity
   {\normalfont (}thus $d$ is equal to~$\varphi(n)$, where $\varphi$ is Euler's totient function{\normalfont )},
   and $\lambda$ be a partition whose length is at most $d$.
   Moreover, we assume the following condition on~$n$:
   \begin{quote}
       {\normalfont ($*$)} $n$ has at most two distinct odd prime factors.
   \end{quote}
   Then, we have
   \[
      s_{\lambda}(\omega_1,\ldots,\omega_d) = 1,0, \text{ or } {-1}.
   \]
   Here, $s_{\lambda}$ is the Schur polynomial associated to~$\lambda$.
\end{theorem}

The condition ($*$) holds for many natural numbers.
For example, all natural numbers less than $105 = 3 \cdot 5 \cdot 7$ satisfy ($*$).

Theorem~\ref{thm:main} has been known for~$\lambda = (1^k)$ and $(k)$
as properties of the coefficients of the cyclotomic polynomial and its  multiplicative inverse.

First, when $\lambda = (1^k)$, 
the Schur polynomial associated with~$\lambda$ equals the $k$th elementary symmetric polynomial $e_k$.
Thus, we have
\[
   s_{\lambda}(\omega_1,\ldots,\omega_d) = e_k(\omega_1,\ldots,\omega_d),
\]
and this equals the coefficient of~$x^{d-k}$ in the cyclotomic polynomial $\Phi_n(x)$
(up to sign),
because
\[
   \Phi_n(x) = (x-\omega_1) \cdots (x-\omega_d).
\]
As is well known,
A. Migotti \cite{Mi} showed that the coefficients of~$\Phi_n(x)$ are all in the set $\{ 1, 0, -1 \}$,
if $n$ satisfies ($*$).

Secondly, when $\lambda = (k)$, 
the Schur polynomial associated with~$\lambda$ equals the $k$th complete homogeneous symmetric polynomial $h_k$.
Thus, we have
\[
   s_{\lambda}(\omega_1,\ldots,\omega_d) = h_k(\omega_1,\ldots,\omega_d),
\]
and this equals the coefficient of~$x^k$ in~$\Phi_n(x)^{-1}$.
Indeed, we have
\begin{align*}
   \Phi_n(x)^{-1}
   &= (x-\omega_1)^{-1} \cdots (x-\omega_d)^{-1} \\
   &= (-)^d \omega_1 \cdots \omega_d (1 - x\omega_1^{-1})^{-1} \cdots (1 - x\omega_d^{-1})^{-1} \\
   &= \sum_{k \geq 0} x^k h_k(\omega_1^{-1},\ldots,\omega_d^{-1}) \\
   &= \sum_{k \geq 0} x^k h_k(\omega_1,\ldots,\omega_d),
\end{align*}
because 
$(-)^d \omega_1 \cdots \omega_d = 1$ and 
$h_k(\omega_1^{-1},\ldots,\omega_d^{-1}) = h_k(\omega_1,\ldots,\omega_d)$.
P. Moree \cite{Mo} showed that the coefficients of~$\Phi_n(x)^{-1}$ are all in the set $\{ 1, 0, -1 \}$,
if $n$ satisfies ($*$)\footnote{%
Interestingly, the proof in the case $\lambda = (k)$ is easier than that in the case $\lambda = (1^k)$.}.

Theorem~\ref{thm:main} is a generalization of these two results.

The key to the proof is the concept of a unimodular system of vectors
(see Section~\ref{sec:unimodular_systems} for the definition).
Namely, Theorem~\ref{thm:main} is reduced to the following theorem:

\begin{theorem}[Proposition~\ref{prop:X_otimes_Y_is_unimodular} (2)] \label{thm:X_otimes_Y_is_unimodular} \slshape
   The tensor product of two maximal circuits is a unimodular system.
\end{theorem}

Here, we call a finite subset of a finite dimensional vector space $V$ a maximal circuit, 
if it can be expressed as $B \cup \{ - \sum B \}$ for some basis $B$ of $V$
(we use this terminology, because such a set forms a maximal circuit as a matroid in $V$).

We note that the tensor product of \textit{three} maximal circuits is not necessarily unimodular.
Thus the number \textit{two} is essential in Theorem~\ref{thm:X_otimes_Y_is_unimodular}.
Moreover the process of attributing Theorem~\ref{thm:main} to Theorem~\ref{thm:X_otimes_Y_is_unimodular}
highlights the origin of the special significance of \textit{two} in Theorem~\ref{thm:main}
(see Section~\ref{sec:propositions_on_the_unimodularity} for the detail).

The proof of Theorem~\ref{thm:main} is quite different from those in the previous studies in \cite{Mi} and \cite{Mo}.
The authors consider Theorem~\ref{thm:main} to be interesting in its own right, 
as is its unexpected connection with unimodular systems.

%
\section{Unimodular systems}
\label{sec:unimodular_systems}
%
%
The key to the proof of Theorem~\ref{thm:main} is the concept of
a unimodular system of vectors \cite{DG}.
Let $K$ be a field of characteristic $0$, and 
$V$ an $n$-dimensional $K$-vector space.

\begin{proposition} \label{prop:unimodular_system}\slshape
   Let $X$ be a finite subset of~$V$ satisfying $0 \not\in X$ and $\langle X \rangle = V$.
   The following conditions are equivalent:\\
   \textup{(1)} For any basis $B \subset X$, the determinant
   does not depend on~$B$.\\
   \textup{(2)} For any basis $B \subset X$, the set $\mathbb{Z}B$ does not depend on~$B$.\\
   \textup{(3)} $X$ can be identified with the columns of a totally unimodular matrix
   through a linear isomorphism $V \to K^n$.
\end{proposition}

We say that $X$ is a unimodular system of~$V$, if one of these three conditions holds.

To understand condition (1) of Proposition~\ref{prop:unimodular_system},
we need to clarify what is meant by the determinant of~$B$.
We can express $B$ as an $n$ by $n$ matrix through a linear isomorphism $f \colon V \to K^n$,
so that we can define $\det B$ excluding the sign through this correspondence
(an ambiguity of sign caused by the order of the elements of~$B$).
Let us put $[a] = \{ a, -a \}$ 
(this is the equivalent class determined by identifying two scalars equal up to sign).
In this way, we can determine $[\det B]$ relative to~$f$.
Condition (1) of Proposition~\ref{prop:unimodular_system} does not depend on the choice of~$f$.

In condition (3), the term ``a totally unimodular matrix''
is used to refer to a matrix for which the determinant of every square submatrix is $1$, $0$, or $-1$.

\begin{remark}
   The original definition of unimodular system in~\cite{DG}
   does not require the conditions $0 \not\in X$ or $\langle X \rangle = V$.
\end{remark}

The root systems of type A are a typical example of a unimodular system
(in fact, this is a maximal unimodular system).

%
\section{Propositions on the unimodularity of vector systems}
\label{sec:propositions_on_the_unimodularity}
%
%
In this section, we state four propositions 
on the unimodularity of vector systems
(Propositions~\ref{prop:omega_n_is_unimodular}--\ref{prop:X_otimes_Y_is_unimodular}).
The main theorem is reduced to these four propositions sequentially:
\begin{align*}
   \text{Theorem~\ref{thm:main}} 
   &\Leftarrow \text{Proposition~\ref{prop:omega_n_is_unimodular}} \\
   &\Leftarrow \text{Proposition~\ref{prop:Z_n_is_unimodular}} \\
   &\Leftarrow \text{Proposition~\ref{prop:Z_p_otimes_Z_q_is_unimodular}} \\
   &\Leftarrow \text{Proposition~\ref{prop:X_otimes_Y_is_unimodular}}.
\end{align*}

The four propositions are as follows:

\begin{proposition}\label{prop:omega_n_is_unimodular} \slshape
   When $n$ satisfies {\normalfont ($*$)}, the following set is a unimodular system of~$\mathbb{C}^d$:
   \[
      \Omega_n = \left\{ \left.
      \begin{pmatrix} \omega_1^k \\ \vdots \\ \omega_d^k \end{pmatrix}
      \,\right|\,
      k = 0,1,\ldots,n-1
      \right\}.
   \]
\end{proposition}

\begin{proposition}\label{prop:Z_n_is_unimodular} \slshape
   When $n$ satisfies {\normalfont ($*$)}, the following set is a unimodular system of~$\mathbb{Q}(\zeta_n)$:
   \[
      Z_n = \big\{ z \in \mathbb{C} \,\big|\, z^n = 1 \big\}.
   \]
\end{proposition}

Here, $\zeta_n$ is a primitive $n$th root of unity,
and we regard the cyclotomic field $\mathbb{Q}(\zeta_n)$ as a $d$-dimensional $\mathbb{Q}$-vector space.

\begin{proposition}\label{prop:Z_p_otimes_Z_q_is_unimodular} \slshape
   \textup{(1)} If $p$ is an odd prime, $Z_p$ is a unimodular system of~$\mathbb{Q}(\zeta_p)$. \\
   \textup{(2)} If $p$ and $q$ are odd primes,
   \[
      Z_p \otimes Z_q = \big\{ x \otimes y \,\big|\, x \in Z_p, \,\, y \in Z_q \big\}
   \]
   is a unimodular system of~$\mathbb{Q}(\zeta_p) \otimes \mathbb{Q}(\zeta_q)$,
   where ``$\otimes$'' means the tensor product of two $\mathbb{Q}$-vector spaces. 
\end{proposition}

\begin{proposition}\label{prop:X_otimes_Y_is_unimodular} \slshape
    \textup{(1)} 
    If $X$ is a maximal circuit of~$V$, then $X$ is a unimodular system of~$V$. \\
    \textup{(2)} 
    If $X$ and $Y$ are maximal circuits of~$V$ and $W$, respectively,
    then
   \[
      X \otimes Y = \big\{ x \otimes y \,\big|\, x \in X, \,\, y \in Y \big\}
   \]
   is a unimodular system of~$V \otimes W$.
\end{proposition}

As will be discussed later, Proposition~\ref{prop:X_otimes_Y_is_unimodular} (1) is almost trivial.
Furthermore Proposition~\ref{prop:X_otimes_Y_is_unimodular} (2) coincides with Theorem~\ref{thm:X_otimes_Y_is_unimodular}.
Therefore, the main theorem is reduced to Theorem~\ref{thm:X_otimes_Y_is_unimodular}.

Here, we define the concept of a maximal circuit as follows:

\begin{definition}
   For a finite subset $X$ of a finite dimensional $\mathbb{Q}$-vector space $V$,
   we say that $X$ is a maximal circuit of~$V$,
   when the following conditions hold:
   \[
      |X| = \dim(V) + 1, \qquad
      \langle X \rangle = V, \qquad
      \sum X = 0.
   \]
   Here, we put $\sum X = \sum_{x \in X} x$.
\end{definition}

For example, if $B$ is a basis of~$V$,
$B \cup \{ -\sum B \}$ is a maximal circuit of~$V$.
Conversely, any maximal circuit of~$V$ can be expressed in this form.
Hence, any two maximal circuits of~$V$ are interchanged by a linear automorphism.

\begin{remark}
The tensor product of \textit{three} maximal circuits is not necessarily unimodular.
Indeed we assume that $X_1, X_2, X_3$ are maximal circuits of $V_1, V_2, V_3$, respectively.
When $\dim V_1 = 2$, $\dim V_2 = 4$, $\dim V_3 = 6$, $X_1 \otimes  X_2 \otimes X_3$ is not unimodular.
This fact corresponds to the fact that $105 = 3 \cdot 5 \cdot 7 = (2+1)(4+1)(6+1)$ is 
the smallest $n$ for which the property ``all coefficients of $\Phi_n(x)$ are 1, 0 , or $-1$'' does not hold.
\end{remark}

Moreover, when $p$ is an odd prime,
$Z_p$ is a maximal circuit of the $\mathbb{Q}$-vector space $\mathbb{Q}(\zeta_p)$.
Hence, Proposition~\ref{prop:X_otimes_Y_is_unimodular}
is a generalization of Proposition~\ref{prop:Z_p_otimes_Z_q_is_unimodular}.

In the next section,
we will explain the reduction of the main theorem to
Theorem~\ref{thm:X_otimes_Y_is_unimodular}
via Propositions~\ref{prop:omega_n_is_unimodular}--\ref{prop:X_otimes_Y_is_unimodular}.
Moreover, we will prove Theorem~\ref{thm:X_otimes_Y_is_unimodular}
in Section~\ref{sec:proof_of_the_last_Prop}.

%
\section{Reduction of the main theorem to Theorem~\ref{thm:X_otimes_Y_is_unimodular}}
%
%
In this section,
we explain the reduction of the main theorem to
Theorem~\ref{thm:X_otimes_Y_is_unimodular}
via propositions stated in the previous section.

\subsection{Theorem~\ref{thm:main} $\Leftarrow$ Proposition~\ref{prop:omega_n_is_unimodular}}
%
First, Theorem~\ref{thm:main} is reduced to Proposition~\ref{prop:omega_n_is_unimodular}.
Indeed, using Proposition~\ref{prop:omega_n_is_unimodular},
we can prove Theorem~\ref{thm:main} as follows.

The Schur polynomial $s_{\lambda}$ is expressed as
\[
   s_{\lambda}(x_1,\ldots,x_d) = a_{\delta + \lambda}(x_1,\ldots,x_d) / a_{\delta}(x_1,\ldots,x_d).
\]
Here, $a_{\mu}$ is a Vandermonde type determinant defined by
\[
   a_{\mu}(x_1,\ldots,x_d) = \det(x_i^{\mu_j})_{1 \leq i,j \leq d}
\]
for~$\mu = (\mu_1,\ldots,\mu_d) \in \mathbb{Z}_{\geq 0}^d$.
Moreover, we put 
\[
   \delta = (d-1,d-2,\ldots,1,0).
\]
For any $\mu \in \mathbb{Z}_{\geq 0}^d$,
there exist $v_1,\ldots,v_d \in \Omega_n$ such that
\[
   a_{\mu}(\omega_1,\ldots,\omega_d) = \det(v_1,\ldots,v_d).
\]
By Proposition~\ref{prop:omega_n_is_unimodular},
when $n$ satisfies ($*$),
there exists a nonzero complex number $a$ satisfying
\[
   \big\{ a_{\mu}(\omega_1,\ldots,\omega_d) \,\big|\, \mu \in \mathbb{Z}_{\geq 0}^d \big\}
   = \big\{ \det(v_1,\ldots,v_d) \,\big|\, v_1,\ldots,v_d \in \Omega_n \big\} 
   = \{ a,0,-a \}.
\]
Moreover, we see $a_{\delta}(\omega_1,\ldots,\omega_d) \ne 0$ easily.
Theorem~\ref{thm:main} is immediate from this.

\subsection{Proposition~\ref{prop:omega_n_is_unimodular} $\Leftarrow$ Proposition~\ref{prop:Z_n_is_unimodular}}
%
Proposition~\ref{prop:omega_n_is_unimodular}
can be reduced to Proposition~\ref{prop:Z_n_is_unimodular}
through a natural linear isomorphism as follows.

To prove Proposition~\ref{prop:omega_n_is_unimodular},
it suffices to show that $\Omega_n$ is a unimodular system of~$\langle \Omega_n \rangle$,
the $\mathbb{Q}$-vector space generated by~$\Omega_n$.
We note that $\langle \Omega_n \rangle$
is isomorphic to the cyclotomic field $\mathbb{Q}(\zeta_n)$
(as $\mathbb{Q}$-vector spaces) through the correspondence
\[
   \begin{pmatrix} z_1 \\ \vdots \\ z_d \end{pmatrix} \mapsto z_1.
\]
Moreover, $\Omega_n$ is identified with~$Z_n$ through this isomorphism.
Thus, Proposition~\ref{prop:omega_n_is_unimodular}
is equivalent to Proposition~\ref{prop:Z_n_is_unimodular}.

\subsection{Proposition~\ref{prop:Z_n_is_unimodular} $\Leftarrow$ Proposition~\ref{prop:Z_p_otimes_Z_q_is_unimodular}}
%
Proposition~\ref{prop:Z_n_is_unimodular} can also be reduced to
Proposition~\ref{prop:Z_p_otimes_Z_q_is_unimodular}
through a natural linear isomorphism.

First we note the following lemma:

\begin{lemma} \slshape
   When $X$ and $Y$ are unimodular systems of~$V$ and $W$, respectively,
   $X \sqcup Y$ is a unimodular system of~$V \oplus W$.
\end{lemma}

Next, 
when we have $n = p_1^{l_1} \cdots p_k^{l_k}$
where $p_1,\ldots,p_k$ are distinct primes,
we can identify $\mathbb{Q}(\zeta_n)$ with
the $p_1^{l_1 - 1} \cdots p_k^{l_k - 1}$-fold direct sum of
\[
   \mathbb{Q}(\zeta_{p_1}) \otimes \cdots \otimes \mathbb{Q}(\zeta_{p_k})
\]
as $\mathbb{Q}$-vector spaces.
Moreover, through this isomorphism, 
we can identify $Z_n$ with
the $p_1^{l_1 - 1} \cdots p_k^{l_k - 1}$-fold disjoint sum of
\[
   Z_{p_1} \otimes \cdots \otimes Z_{p_k}.
\]
This follows from (1) and (2) of the following lemma:

\begin{lemma} \slshape
   \textup{(1)} 
   When natural numbers $a$ and $b$ are coprime,
   there exists a linear isomorphism
   $\mathbb{Q}(\zeta_{ab}) \to \mathbb{Q}(\zeta_a) \otimes \mathbb{Q}(\zeta_b)$
   such that the image of~$Z_{ab}$ is equal to~$Z_a \otimes Z_b$. \\
   \textup{(2)}
   For any prime $p$,
   there exists a linear isomorphism 
   $\mathbb{Q}(\zeta_{p^l}) \to \mathbb{Q}(\zeta_p)^{\oplus p^{l-1}}$ such that
   the image of~$Z_{p^l}$ is equal to the $p^{l-1}$-fold disjoint sum of~$Z_p$.
\end{lemma}

\begin{proof}[Proof.]
(1) We consider the following correspondence:
\[
   \mathbb{Q}(\zeta_a) \otimes \mathbb{Q}(\zeta_b) \to \mathbb{Q}(\zeta_{ab}), \qquad
   z \otimes w \mapsto zw.
\] 
This gives a linear isomorphism,
and the image of~$Z_a \otimes Z_b$ is equal to~$Z_{ab}$. \\
(2) Let $\mathbb{Q}(\zeta_p)^{(j)}$ denote a copy of~$\mathbb{Q}(\zeta_p)$
for~$j \in 0,1,\ldots, p^{l-1}-1$.
Moreover, we denote by~$z^{(j)}$ the counterpart of~$z \in \mathbb{Q}(\zeta_p)$ in~$\mathbb{Q}(\zeta_p)^{(j)}$.
Let us consider the following correspondence:
\[
   \bigoplus_{j=0}^{p^{l-1}-1}\mathbb{Q}(\zeta_p)^{(j)}
   \to \mathbb{Q}(\zeta_{p^l}), \qquad
   z^{(j)}\mapsto \zeta_{p^l}^j z.
\]
This gives a linear isomorphism
(it suffices to show the surjectiveness because the dimensions are equal).
It is obvious that the image of~$\bigsqcup_{j=0}^{p^{l-1}-1}Z_p^{(j)}$ is equal to~$Z_{p^l}$.
\end{proof}

Thus, we have the following isomorphism,
because $\mathbb{Q}(\zeta_2) = \mathbb{Q}$: 
\[
	\mathbb{Q}(\zeta_n) \simeq
	\begin{cases}
		\bigoplus_{i\in \Lambda}\mathbb{Q}^{(i)}
		\ (\textup{where }|\Lambda|=2^{k-1}), &n=2^k, \\[5pt]
		\bigoplus_{i\in \Lambda}\mathbb{Q}(\zeta_p)^{(i)}
		\ (\textup{where }|\Lambda|=p^{l-1}), &n=p^l, \\[5pt]
		\bigoplus_{i\in \Lambda}\mathbb{Q}(\zeta_p)^{(i)}
		\ (\textup{where }|\Lambda|=2^{k-1}p^{l-1}), &n=2^kp^l, \\[5pt]
		\bigoplus_{i\in\Lambda}(\mathbb{Q}(\zeta_p)\otimes \mathbb{Q}(\zeta_q))^{(i)}
		\ (\textup{where }|\Lambda|=p^{l-1}q^{m-1}), &n=p^lq^m, \\[5pt]
		\bigoplus_{i\in\Lambda}(\mathbb{Q}(\zeta_p)\otimes \mathbb{Q}(\zeta_q))^{(i)}
		\ (\textup{where }|\Lambda|=2^{k-1}p^{l-1}q^{m-1}), &n=2^kp^lq^m.
	\end{cases}
\]
Here, $p$ and $q$ are distinct odd primes.
Through this isomorphism $f$, we can write the image $f(Z_n)$ as
\[
	f(Z_n) =
	\begin{cases}
		\bigsqcup_{i\in\Lambda}Z_2^{(i)}, &n=2^k, \\[5pt]
		\bigsqcup_{i\in\Lambda}Z_p^{(i)}, &n=p^l, \\[5pt]
		\bigsqcup_{i\in \Lambda}(Z_2 \otimes Z_p)^{(i)}, &n=2^kp^l, \\[5pt]
		\bigsqcup_{i\in\Lambda}(Z_p\otimes Z_q)^{(i)}, &n=p^lq^m, \\[5pt]
		\bigsqcup_{i\in\Lambda}(Z_2 \otimes Z_p\otimes Z_q)^{(i)}, &n=2^kp^lq^m.
	\end{cases}
\]
Thus, we see that 
Proposition~\ref{prop:Z_n_is_unimodular} follows from
Proposition~\ref{prop:Z_p_otimes_Z_q_is_unimodular}.
Indeed, we have $Z_2 = \{ 1,-1 \}$ and the following lemma:

\begin{lemma} \slshape
   If $X$ is a unimodular system of~$V$,
   then $\{ 1,-1\} \otimes X$ is also a unimodular system of~$V$.
\end{lemma}

\subsection{Proposition~\ref{prop:Z_p_otimes_Z_q_is_unimodular} $\Leftarrow$ Proposition~\ref{prop:X_otimes_Y_is_unimodular}}
%
When $p$ is an odd prime,
$Z_p$ is a maximal circuit of the $\mathbb{Q}$-vector space $\mathbb{Q}(\zeta_p)$.
Thus, Proposition~\ref{prop:X_otimes_Y_is_unimodular}
is a generalization of
Proposition~\ref{prop:Z_p_otimes_Z_q_is_unimodular}.

%
\section{Proof of Theorem~\ref{thm:X_otimes_Y_is_unimodular}}
\label{sec:proof_of_the_last_Prop}
%
%
The main theorem has been reduced to Proposition~\ref{prop:X_otimes_Y_is_unimodular}.
In this section, we prove it.
Since Proposition~\ref{prop:X_otimes_Y_is_unimodular} (1) is almost trivial,
the main task is to prove Proposition~\ref{prop:X_otimes_Y_is_unimodular} (2), 
that is, Theorem~\ref{thm:X_otimes_Y_is_unimodular}.

\subsection{Every maximal circuit is unimodular}\label{subsec:proof_of_(1)}
%
First, we prove Proposition~\ref{prop:X_otimes_Y_is_unimodular} (1).
Namely, we show that every maximal circuit is unimodular.

\begin{proof}[Proof of Proposition~\textsl{\ref{prop:X_otimes_Y_is_unimodular}} \textsl{(1)}.]
It suffices to show that the following matrix is totally unimodular:
\[
   \begin{pmatrix}
   -1 & 1 & 0 & \ldots & 0 \\
   -1 & 0 & 1 & \ldots & 0  \\
   \vdots & \vdots & & \ddots & \vdots \\
   -1 & 0 & 0 & \ldots & 1  
   \end{pmatrix}.
\]
This is immediate from the following lemma.
\end{proof}

\begin{lemma} \label{lem:(A I)_is_totally_unimodular}\slshape
   If an $m$ times $n$ matrix $A$ is totally unimodular, 
   then $(A \,\,\, I_m)$ is also totally unimodular.
   Here $I_m$ is the unit matrix of size $m$.
\end{lemma}

\subsection{The tensor product of two maximal circuits is unimodular}\label{subsec:proof_of_(2)}
%
Next, we prove Proposition~\ref{prop:X_otimes_Y_is_unimodular} (2),
namely Theorem~\ref{thm:X_otimes_Y_is_unimodular} using network matrices.

Let us explain the concept of a network matrix \cite{S,T}.
For a directed tree $(\mathcal{V}, \mathcal{T})$ and
a directed graph $(\mathcal{V}, \mathcal{E})$,
we define the $\mathcal{T} \times \mathcal{E}$ matrix $M = (M_{x,y})$ as follows.
For $x \in \mathcal{T}$ and $y = (u, v) \in \mathcal{E}$,
we put
\[
   M_{x,y} = 
   \begin{cases}
      1, & \text{if $b$ occurs in forward direction in $P$}, \\
      -1,& \text{if $b$ occurs in backward direction in $P$}, \\
      0,& \text{if $b$ does not occur in $P$},
   \end{cases}
\]
where $P$ be the unique undirected path from $u$ to $v$ in $\mathcal{T}$. 
Then the matrix $M$ is called the network matrix
represented by $(\mathcal{V}, \mathcal{T})$ and $(\mathcal{V}, \mathcal{E})$.
In general, network matrices are known to be totally unimodular.

\begin{proof}[Proof of Theorem~\textsl{\ref{thm:X_otimes_Y_is_unimodular}}.]
Put $m = |X| - 1$ and $n = |Y| - 1$.
Let us denote the elements of $X$ by $e_0, e_1,\ldots,e_m$,
and the elements of $Y$ by $f_0, f_1,\ldots,f_n$:
\[
   X = \{ e_0, e_1,\ldots,e_m \}, \qquad
   Y = \{ f_0, f_1,\ldots,f_n \}.
\]
Without loss of generality,
we can assume that
\[
   X_+ = \{ e_1,\ldots,e_m \}, \qquad 
   Y_+ = \{ f_1,\ldots,f_n \}
\]
are the standard bases of $V = \mathbb{Q}^m$ and $W = \mathbb{Q}^n$, respectively.
Noting Lemma~\ref{lem:(A I)_is_totally_unimodular}, 
we see that it suffices to prove that 
\[
   R = 
   \{ e_0 \otimes f_0 \} \sqcup
   (e_0 \otimes Y_+) \sqcup 
   (X_+ \otimes f_0)
\]
is a unimodular system.
Indeed $X_+ \otimes Y_+$ can be identified with the unit matrix $I_{mn}$.
We denote the $mn \times (m+n+1)$ matrix corresponding to $R$ by $A$.
For example, when $(m,n)=(2,3)$, we have
\[
   A = 
   \left(
   \begin{array}{c|ccc|cc}
   1 & -1 & 0 & 0 & -1 & 0 \\
   1 & -1 & 0 & 0 & 0 & -1 \\
   1 & 0 & -1 & 0 & -1 & 0 \\
   1 & 0 & -1 & 0 & 0 & -1 \\
   1 & 0 & 0 & -1 & -1 & 0 \\
   1 & 0 & 0 & -1 & 0 & -1 \\
   \end{array}
   \right).
\]

Then, the transposed matrix ${}^t\!A$ is equal to 
the network matrix represented by $(\mathcal{V},\mathcal{T})$ and $(\mathcal{V},\mathcal{E})$.
Here we put
\begin{align*}
   \mathcal{V} &= \{ e_0,e_1,\ldots,e_m \} \cup \{ f_0,f_1,\ldots,f_n \}, \\
   \mathcal{T} &=
   \big\{ (e_0, f_0) \big\} \cup
   \big\{ (e_0, f_j) \,\big|\, 1 \leq j \leq n \big\} \cup
   \big\{ (e_i, f_0) \,\big|\, 1 \leq i \leq m \big\}, \\
   \mathcal{E} &= \{ (e_i, f_j) \,|\, 1 \leq i \leq m, \,\, 1 \leq j \leq n \}.
\end{align*}
Namely $\mathcal{E}$ is the complete bipartite graph with bipartition $(\{ e_1,\ldots,e_m \}, \{ f_1,\ldots,f_n \})$.
Hence ${}^t\!A$ is totally unimodular, and so is $A$.
This means the unimodularity of $R$, and therefore of $X_+ \otimes Y_+$.
\end{proof}

\begin{remark}
We note that slightly weaker result than Theorem~\ref{thm:X_otimes_Y_is_unimodular} can be derived from matroid theory.
Namely, when vector systems $X$ and $Y$ are maximal circuits,
$X \otimes Y$ is isomorphic to a unimodular system as matroids.
This follows from this following two facts: 
\begin{itemize}
   \item When $X$ and $Y$ are maximal circuits,
   $X \otimes Y$ is isomorphic 
   to the cographic matroid determined by the complete bipartite graph with bipartition $(X, Y)$.
   \item Both graphic matroids and cographic matroids are regular matroids, 
   and every regular matroid can be realized by a unimoudlar system \cite{O}.
\end{itemize}
Our Theorem~\ref{thm:X_otimes_Y_is_unimodular} is stronger than this,
because it asserts that $X \otimes Y$ \textit{itself} is unimodular.

We also note that the root system of type $\mathrm{A}_n$ is 
isomorphic to the graphic matroid determined by the complete graph $K_{n+1}$
(by identifying $v$ with~$-v$).
Thus, the tensor product of two maximal circuits
can be regarded as an analogue of the root system of type $\mathrm{A}_n$
in the framework of bipartite graphs.
\end{remark}

\begin{remark}
This proof of Theorem~\ref{thm:X_otimes_Y_is_unimodular} was greatly simplified
following the reviewer's suggestion after the first version of this article was submitted.
The reviewer pointed out that a simpler proof using network matrices exists.
The previous proof of Theorem~\ref{thm:X_otimes_Y_is_unimodular} was 
based on an argument in the complete bipartite graph.
The previous proof is long, but
it is worth noting that the unimodularity of the root system of type~$\mathrm{A}_n$
can similarly be proved by an argument in the complete graph $K_{n+1}$.
\end{remark}

%
\section*{Acknowledgements}
%
%
The authors would like to thank the reviewers for their valuable comments.
Thanks to their suggestions, the proof was greatly simplified,
and the importance of Theorem~\ref{thm:X_otimes_Y_is_unimodular} has been more clearly explained.

%
%
%


\begin{thebibliography}{DG}

\bibitem[DG]{DG}
V. Danilov and V. Grishukhin,
\textit{Maximal unimodular systems of vectors},
European J. Combin. \textbf{20} (1999), no.\ 6, 507--526. 


\bibitem[H]{H}
M. Hidaka,
\textit{The Schur polynomials in primitive $n$th roots of unity} (in Japanese),
Master's thesis, Kagoshima University, 2019.

\bibitem[Mi]{Mi}
A. Migotti, 
\textit{Zur Theorie der Kreisteilungsgleichung}, 
S.-B. der Math.-Naturwiss. Classe der Kaiser. Akad. der Wiss., Wien \textbf{87} (1883), 7--14.

\bibitem[Mo]{Mo}
P. Moree, 
\textit{Inverse cyclotomic polynomials}, 
J. Number Theory \textbf{129} (2009), 667--680.

\bibitem[O]{O}
J. Oxley, 
Matroid theory. Second edition,
Oxford Graduate Texts in Mathematics, 21. 
Oxford University Press, Oxford, 2011.

\bibitem[S]{S}
A. Schrijver,
Theory of linear and integer programming,
John Wiley \& Sons, 1998.

\bibitem[T]{T}
W. T. Tutte,
\textit{Lectures on matroids},
J. Res. Natl. Bur. Stand., Sect. B \textbf{69} (1965), 1--47.

\end{thebibliography}
\end{document}